\documentclass[reqno]{amsart}
\usepackage{philstyle}
\usepackage{amssymb,enumerate}

\newcommand{\cA}{{\mathcal{A}}}
\newcommand{\cR}{{\mathcal{R}}}

\newcommand{\hB}{{\widehat{B}}}
\newcommand{\hD}{{\widehat{D}}}
\newcommand{\hf}{{\widehat{f}}}
\newcommand{\hg}{{\widehat{g}}}
\newcommand{\hH}{{\widehat{H}}}
\newcommand{\hI}{{\widehat{I}}}
\newcommand{\hP}{{\widehat{P}}}
\newcommand{\hT}{{\widehat{T}}}
\newcommand{\hX}{{\widehat{X}}}
\newcommand{\hmu}{{\widehat{\mu}}}
\newcommand{\hnu}{{\widehat{\nu}}}
\newcommand{\hPi}{{\widehat{\Pi}}}

\newcommand{\N}{{\mathbb N}}
\newcommand{\R}{{\mathbb R}}

\newcommand{\bx}{{\mathbf{x}}}
\newcommand{\by}{{\mathbf{y}}}
\newcommand{\bz}{{\mathbf{z}}}

\newcommand{\barf}{\overline{f}}
\newcommand{\barg}{\overline{g}}

\newcommand{\tmut}{{{\mu^{\scriptscriptstyle{P}}_t}}}
\newcommand{\tmuti}{{{\mu^{\scriptscriptstyle{P}}_{t_i}}}}
\newcommand{\tmutst}{{{\mu^{\scriptscriptstyle{P}}_{t_*}}}}
\newcommand{\thmut}{{\widehat{\tmut}}}
\newcommand{\thmuti}{{\widehat{\tmuti}}}
\newcommand{\thmutst}{{\widehat{\tmutst}}}

\DeclareMathOperator{\Bd}{Bd}

\newcommand{\ilim}{{\varprojlim}}
\newcommand{\thr}[1]{\left< #1 \right>}
\newcommand{\floor}[1]{\left\lfloor #1 \right\rfloor}

\begin{document}

\title[Statistical Stability -- Tent map attractors]{Statistical Stability
for Barge-Martin attractors derived from tent maps}

\author{Philip Boyland}
\address{Department of Mathematics\\University of Florida\\372 Little
    Hall\\Gainesville\\ FL 32611-8105, USA}
\email{boyland@ufl.edu}
\author{Andr\'e de Carvalho}
\address{Departamento de Matem\'atica Aplicada\\ IME-USP\\ Rua Do Mat\~ao
    1010\\ Cidade Universit\'aria\\ 05508-090 S\~ao Paulo SP\\ Brazil}
\email{andre@ime.usp.br}
\author{Toby Hall}
\address{Department of Mathematical Sciences\\ University of Liverpool\\
    Liverpool L69 7ZL, UK}
\email{tobyhall@liverpool.ac.uk}

\date{December 2019}
\thanks{The authors are grateful for the support of FAPESP grant
  \mbox{2016/25053-8} and CAPES grant 88881.119100/2016-01}

\begin{abstract} 
	Let $\{f_t\}_{t\in(1,2]}$ be the family of core tent maps of slopes~$t$.
	The parameterized Barge-Martin construction yields a family of disk
	homeomorphisms $\Phi_t\colon D^2\to D^2$, having transitive global
	attractors $\Lambda_t$ on which $\Phi_t$ is topologically conjugate to the
	natural extension of $f_t$. The unique family of absolutely continuous
	invariant measures for $f_t$ induces a family of ergodic $\Phi_t$-invariant
	measures $\nu_t$, supported on the attractors~$\Lambda_t$.

	We show that this family $\nu_t$ varies weakly continuously, and that the
	measures $\nu_t$ are physical with respect to a weakly continuously varying
	family of background Oxtoby-Ulam measures $\rho_t$. 

    Similar results are obtained for the family $\chi_t\colon S^2\to S^2$ of
    transitive sphere homeomorphisms, constructed in a previous paper of the
    authors as factors of the natural extensions of~$f_t$.
\end{abstract}

\maketitle


\section{Introduction}
The parameterized Barge-Martin construction applied to the family~$f_t$ of core
tent maps on the interval~$I$ yields a family of disk homeomorphisms $\Phi_t$,
each of which has a transitive global attractor $\Lambda_t$ which is
homeomorphic to the inverse limit $\ilim(I, f_t)$ and on which $\Phi_t$ is
topologically conjugate to the natural extension $\hf_t$~\cite{BM, barge,
family, prime}. These inverse limits are topologically quite intricate and have
been the subject of intense investigation. The main recent focus  has been the
proof of the Ingram conjecture: for different values of $t\in [\sqrt{2}, 2]$
the inverse limits are not homeomorphic (for references see~\cite{Ingram},
which contains the final proof for tent maps which are not restricted to their
cores). In the case  where the parameter $t$ is such that the critical orbit of
$f_t$ is dense (a full measure, dense $G_\delta$ set of parameters), theorems
of Bruin and of Raines imply that the inverse limit $\hI_t$ is nowhere locally
the product of a Cantor set and an interval~\cite{bruin,raines}. Perhaps more
striking, Barge, Brooks and Diamond~\cite{BBD} show that there is a
dense~$G_\delta$ set of parameters~$t$ for which the inverse limit has a strong
self-similarity: every open subset of~$\hI_t$ contains a homeomorphic copy of
$\hI_s$ for every $s\in[\sqrt{2}, 2]$.

A parameterized family of dynamical systems is said to be \emph{statistically
stable} if the time averages of observables along typical orbits vary
continuously with the parameter: more formally, if there is a family of
invariant physical measures which vary continuously in the weak topology. The
classic examples of systems statistically stable under perturbation are
expanding smooth maps and Axiom A diffeomorphisms restricted to basins of
attractors. There have been extensions into the non-uniformly hyperbolic
context~\cite{ss1, ss2, ss3}, and a variety of specific, mostly low
dimensional, families~\cite{ss4, ss5, ss6, ss7, ss8} have been shown to be
statistically stable; see~\cite{sssurvey} for a survey. Of central importance
here  is a recent result of Alves and Pumari\~no~\cite{AP}, who prove that the
core tent map family $f_t\colon I\to I$ is statistically stable in the stronger
sense that the densities of its family of unique absolutely continuous
invariant measures (acims) $\mu_t$ vary continuously in $L^1$. These  acims
$\mu_t$ on~$I$ for $f_t$ induce a family of ergodic  $\Phi_t$-invariant
measures $\nu_t$ on~$D^2$, supported on the Barge-Martin
attractors~$\Lambda_t$. The main result here, Theorem~\ref{main}, states that
this family of $\Phi_t$-invariant measures $\nu_t$ is weakly continuous.

The definition of statistical stability also requires that the measures $\nu_t$
be physical: the set of regular points of each 
$\nu$ should have positive measure
with respect to a background Lebesgue measure. Equivalently,
if $\delta_z$ denotes the Dirac measure at the point $z$, 
then when $\nu$ is physical
\begin{equation*}
\frac{1}{N}\sum_{k=0}^{N-1} \delta_{f^k(x)} \rightarrow \nu
\end{equation*}
in the weak topology for a set of initial conditions $x$
of positive Lebesgue probability.
This means that the invariant measure $\nu$ is physically observable.
In the case of planar attractors considered here, the existence of  
physical (SRB\footnote{See Remark~\ref{SRB} below}) measures on the
attractors have been proved in a variety of cases; see \cite{Ysurvey, Vsurvey}
for  surveys. 
These include the classical hyperbolic results and their many
non-uniformly hyperbolic extensions as well as the  class of attractors
defined by Wang and Young \cite{WY}.   In view
of the role of inverse limits of unimodal interval maps as 
topological  models for certain H\'enon attractors \cite{barge, BH} 
we note that for the set of {B}enedicks-{C}arleson parameters
Benedicks and Young \cite{BY} have shown the
existence of an SRB measure for the {H}\'enon attractors and
statistical stability at these parameters is proved in \cite{ss4}.

The Barge-Martin construction
rests fundamentally on an inverse limit and there is a natural correspondence
between invariant measures for a map and those of its inverse limit. In
addition, a key result of Kennedy, Raines and Stockman connects regular
points of an invariant measure to those of the corresponding invariant measure
in the inverse limit (\cite{KRS}). However, the background Lebesgue measure on,
say, a manifold does not in any natural way produce such a measure on an
inverse limit. By using a construction from \cite{prime} we can
produce a background physical measure of Oxtoby-Ulam type for the Barge-Martin
homeomorphisms here.  In Section~\ref{sec:iso-meas} we comment on how this can
be used with the Homeomorphic Measures Theorem to obtain a background Lebesgue
measure.

The first paragraph of the following theorem summarizes relevant results
from~\cite{family}, while the second paragraph contains the main results of
this paper.
\begin{theorem}
\label{introthm} There exists a continuously varying family of
    homeomorphisms $\Phi_t:D^2\to D^2$, each having a transitive global
    attractor $\Lambda_t$ which is homeomorphic to the core tent map inverse
    limit $\ilim(I, f_t)$ and on which~$\Phi_t$ is topologically conjugate to
    the natural extension $\hf_t$. The attractors~$\Lambda_t$ vary Hausdorff
    continuously.

    The family $\nu_t$ of ergodic $\Phi_t$-invariant measures induced by the
    acims $\mu_t$ for~$f_t$ is weakly continuous. The measures $\nu_t$ are
    physical with respect to a weakly continuous family of background
    Oxtoby-Ulam measures~$\rho_t$.
\end{theorem}

In~\cite{prime} it is shown how a family $\{\chi_t\}_{t\in(\sqrt2, 2]}$ of
transitive sphere homeomorphisms can be constructed from the core tent maps
$f_t$, with each $\chi_t$ being a factor of the natural extension $\hf_t$ by a
mild semi-conjugacy. We finish the paper by indicating how the same techniques
can be applied to this family.

\section{Preliminaries}
\subsection{Inverse limits}
Let~$X$ be a compact metric space with metric~$d$, and let $g\colon X\to X$ be
continuous and surjective. The \emph{inverse limit} of $g\colon X\to X$ is the
space
\[
    \hX := \ilim(X,g) = \{\bx\in X^\N\,:\,g(x_{n+1})=x_n \text{ for all
     }n\in\N\}.
\] 
We endow $\hX$ with a metric, also denoted~$d$, defined by $d(\bx,
 \by) = \sum_{n=0}^\infty d(x_n, y_n)/2^n$, which induces its natural topology
 as a subspace of the product $X^\N$.

We denote elements of $\hX$ with angle brackets, $\bx =
\thr{x_0,x_1,x_2,\ldots}$ and refer to them as \emph{threads}.
The \emph{projections} are the maps $\pi_n:\hX\to X$ given by
$\pi_n(\thr{x_0,x_1,x_2,\ldots}) = x_n$. The \emph{natural extension} of
$g\colon X\to X$ is the homeomorphism $\hg\colon \hX\to\hX$ defined by
\[
    \hg(\thr{x_0,x_1,x_2,\ldots}) =
    \thr{g(x_0), x_0, x_1, x_2,\ldots}.
\]

Each $g$-invariant Borel probability measure $\nu$ gives rise to a unique
$\hg$-invariant Borel probability measure~$\hnu$ characterized by $(\pi_n)_*
\hnu = \nu$ for all $n$. The measure $\hnu$ is ergodic if and only if the
measure $\nu$ is.

\subsection{The parameterized Barge-Martin construction}
\label{BBM}
We outline the construction here: for details see \cite{BM,family,prime}. Let
$\{f_t\}_{t\in J}$ be the family of core tent maps with slope $t$, affinely
rescaled so they are all defined on the same interval $I = [-1,1]$: that is,
$f_t$ has slope~$t$ on $[-1, c)$ and slope~$-t$ on $(c, 1]$, where $c=1-2/t$,
so that $f_t(x) = \min(t(x-1)+3,\, t(1-x)-1)$. Here~$J=(1,2]$ is the interval
of parameters for tent maps with positive topological entropy.

We will use special coordinates for the disk. Let $S$ be the circle of radius
$2$ centered at the origin of~$\R^2$ with angular coordinates $y$, and let $I$
be the interval $[-1,1]$ in the real axis. Define $G:S\to I$ by $G(y) =
\cos(y)$, and let $D$ be a smooth embedding of the mapping cylinder of $G$.
Thus we may view $D$ as the disk of radius $2$ whose boundary is the circle
$S$. It is decomposed into a family of smoothly varying, smooth arcs
$\eta:S\times [0,1]
\to D$ with each $\eta(y,s)$ connecting the point $\eta(y,0) = y\in S$ to the
point $\eta(y,1) = G(y)\in I$. The point $\eta(y,s)$ of $D$ is given
coordinates $(y,s)$. Thus the interval $I$ consists of all points $(y,1)$ and
points in the interior of $I$ have two coordinates. There is a projection
$\tau\colon S\to I$ defined $\tau(y) = (y,1)$.

Let $\Upsilon\colon D\to D$ be the near-homeomorphism 
defined by
  \[
    \Upsilon(y,s) =
    \begin{cases} 
      (y,2s) & \text{ if } s\in[0,1/2],\\ 
      (y,1) & \text{ if }s\in[1/2,1].
    \end{cases}
  \]
An {\em unwrapping} of the tent map family~$\{f_t\}$ is a continuously varying
family of orientation-preserving near-homeomorphisms $\barf_t\colon D\to D$
with the properties that, for each~$t$,
\begin{enumerate}[a)]
    \item $\barf_t$ is injective on~$I$, and $\barf_t(I) \subseteq
        \{(y,s)\,:\, s\ge 1/2\}$,
  \item $\Upsilon\circ \barf_t|_I = f_t$, and
  \item  for all $y\in S$ and all $s\in[0,1/2]$,
        the second component of $\barf_t(y,s)$ is~$s$.
\end{enumerate}

Given such an unwrapping, let $H_t=\Upsilon \circ \barf_t\colon D\to D$, which
is a near-homeomorphism since both $\Upsilon$ and~$\barf_t$ are. By a theorem
of Brown (\cite{brown}) this implies that the inverse limit $\hD_t := \ilim(D,
H_t)$ is also a topological disk.

Since $H_t = f_t$ on $I$, there is a copy of $\hI_t:=\ilim(I, f_t)$ (which copy
we will also denote $\hI_t$) canonically embedded in $\hD_t$, on which the
restriction of $\hH_t$ agrees with $\hf_t$. Moreover, for every $z$ in the
interior of~$D$ there is some~$N$ such that $H_t^n(z)\in I$ for all $n\ge N$:
therefore any $\bz$ in the interior of $\hD_t$ satisfies $\hH_t^n(\bz)\to\hI_t$
as $n\to\infty$, so that $\hI_t$ is a global attractor for $\hH_t$.

We work simultaneously with the maps $\{H_t\}_{t\in J}$ by collecting them into
a \emph{fat map} $H\colon \Pi := D\times J \to D$ defined by $H(z, t) =
(H_t(z), t)$. The inverse limit $\hPi := \ilim(\Pi, H)$ provides a way
of topologizing the disjoint union of the individual inverse limits $\hD_t$,
since for each~$t$ the map $\iota_t\colon\bz\mapsto \thr{(z_0, t), (z_1, t),
\ldots}$ is an embedding of $\hD_t$ into $\hPi$, and $\hPi$
is the disjoint union of the images of these embeddings as~$t$ varies
through~$J$. We will identify each $\hD_t$ with its image under this embedding
without further comment, thus regarding it as a subset of $\hPi$. Then each
$\hD_t$ is invariant under the natural extension~$\hH$ of $H$, which acts on it
as $\hH_t$.
 
A parameterized version~\cite{family} of Brown's theorem yields a continuous
map $h\colon \hPi\to D^2$, where $D^2$ is a standard model of the
disk, with the property that $h_t := h|_{\hD_t}$ is a homeomorphism onto $D^2$
for each~$t$. The Barge-Martin family of disk homeomorphisms $\{\Phi_t\}_{t\in
J}$ is then defined by $\Phi_t = h_t\circ \hH\circ h_t^{-1}\colon D^2\to D^2$.
They have global attractors $\Lambda_t := h_t(\hI_t)$, which vary Hausdorff
continuously, and on which $\Phi_t$ is topologically conjugate to $\hf_t$.

The ergodic acim $\mu_t$ for $f_t$ on $I$ induces an ergodic $\hf_t$-invariant
measure $\hmu_t$ on $\hI_t$. This in turn generates $\nu_t
:=(h_t)_* \hmu_t$ which is an ergodic $\Phi_t$-invariant measure on $D^2$
supported on the attractor $\Lambda_t$.

\subsection{Weak continuity}
Let $X$ be a compact metric space. A sequence of Borel probability measures
$\mu_n$ is said to \emph{converge weakly} to $\mu_0$ if for all $\alpha\in
C(X, \R)$,
\begin{equation*}
    \int_X \alpha\; d\mu_n\to \int_X \alpha\; d\mu_0.
\end{equation*}
Since this is the only notion of convergence of measures which we will use, the
notation $\mu_n\to \mu_0$ will always denote weak convergence.

The following criteria from pages~16 and~17 of Billingsley~\cite{bill} will be
 used. Recall that a Borel set $A\subset X$ is called a \emph{continuity set}
 for a measure $\mu$ if $\mu(\Bd(A)) = 0$.
\begin{theorem} 
\label{billthm}\mbox{}
	\begin{enumerate}[(a)]
        \item $\mu_n\to \mu_0$ if and only if $\mu_n(A)\to \mu_0(A)$ for all
        $\mu_0$-continuity sets $A$.

        \item Let $\cA$ be a collection of Borel sets which is closed under
        finite intersections, with the property that every open subset of $X$
        is the countable union of sets from $\cA$. If $\mu_n(A)\to \mu_0(A)$
        for all $A\in \cA$, then $\mu_n\to \mu_0$.
	\end{enumerate}
\end{theorem}

A one-parameter family of measures $\mu_t$ is said to be \emph{weakly
continuous} if $t_i\to t$ implies $\mu_{t_i}\to \mu_{t}$. Our starting point in
this paper is a recent result of Alves and Pumari\~no~\cite{AP}, who show that
the family of acims $\mu_t$ for the core tent maps~$f_t$ have densities which
vary continuously in $L^1$, so that, in particular:
\begin{theorem}[Alves \& Pumari\~no] 
    The family $\mu_t$ of acims for the core tent family is weakly continuous.
\end{theorem}

\section{Weak continuity on the fat inverse limit}
In this section we show (Theorem~\ref{measmain}) that weak continuity of the
family~$\mu_t$ of tent map acims implies weak continuity of the induced
measures $\hmu_t$ on the inverse limits $\hI_t$. Since the spaces $\hI_t$ are
varying, this statement needs to be interpreted in the context of fat maps.

Write $P:=I\times J$. Given $X\subset P$ and $t\in J$, write $X^{(t)} := X \cap
(I\times\{t\})$. For each $t\in J$, let $\sigma_t\colon I\to P$ denote the
embedding $x\mapsto (x,t)$, and let $\tmut = (\sigma_t)_*\mu_t$, so that
$\tmut$ is a measure on~$P$ which is supported on~$P^{(t)}$.

It is routine to prove weak continuity of the family of measures
$\{\tmut\}$ (indeed, if~$X$ is any compact metric space and $\{\mu_t\}_{t\in
J}$ is a weakly continuously varying family of measures on~$X$, then the
corresponding measures $\tmut$ on $X\times J$ also vary weakly continuously).

Now let $F\colon P\to P$ be the fat map defined by $F(x,t) = (f_t(x), t)$, and
set $\hP:=\ilim(P, F)$. As in Section~\ref{BBM}, we may think of $\hP$ as the
topologized disjoint union of the inverse limits $\hI_t$, since each map
$\iota_t\colon \hI_t\to \hP$ defined by $\iota_t(\thr{x_0, x_1,  \dots}) =
\thr{(x_0, t), (x_1, t), \dots}$ is a homeomorphism onto its image, which image
is called the \emph{$t$-slice} of~$\hP$.

Recall that $\hmu_t$ denotes the $\hf_t$-invariant measure on $\hI_t$ induced
by the acim $\mu_t$ for~$f_t$. For each~$t$, we write $\thmut =
(\iota_t)_*(\hmu_t)$, a measure on $\hP$ supported on its $t$-slice. Our aim in
this section is to show that the family of measures $\thmut$ is weakly
continuous. 

\begin{remark} 
\label{rk1}
    Let $\pi_n\colon \hP\to P$ be the projections on~$\hP$, and denote by
    $\pi_{n,t}\colon \hI_t\to I$ the projections on $\hI_t$. Then
    $\sigma_t\circ\pi_{n,t} = \pi_n\circ\iota_t$: that is, $\pi_n$ acts as
    $\pi_{n,t}$ on the $t$-slice of $\hP$. It follows that for each $t\in J$ we
    have $(\pi_n)_*\thmut = (\pi_n\circ\iota_t)_*\hmu_t = (\sigma_t\circ
    \pi_{n,t})_* \hmu_t = (\sigma_t)_*\mu_t = \tmut$ for all~$n$, so that
    $\thmut$ is the measure on $\hP$ induced by the measure $\tmut$ on~$P$.
\end{remark}

To show weak continuity of the family $\thmut$, we will apply
Theorem~\ref{billthm}(b) to the family of subsets of~$\hP$ obtained as finite
intersections of sets of the form $\pi_n^{-1}(R)$ for \emph{tilted
rectangles}~$R$, which we now define.

\begin{definition}
    A \emph{tilted rectangle} in $\R^2$ is obtained by rotating an open (possibly empty) rectangle $(a_1, a_2)\times (b_1, b_2)$ by $\pi/4$. We write $\cR$ for the set of all tilted rectangles contained in~$P$.
\end{definition}

\begin{definition}
    A subset $B$ of~$P$ is called \emph{simple} if it satisfies:
    \begin{enumerate}[(a)]
        \item $B$ is open in~$P$, and
        \item for each $t\in J$, $(\Bd(B))^{(t)}$ is finite.
    \end{enumerate}
\end{definition}
Note that, by~(b), a simple subset of~$P$ is a continuity set for all of the measures $\tmut$.

\begin{lemma} \mbox{}
\label{simple-lem}
    \begin{enumerate}[(a)]
        \item Any finite intersection of simple subsets of~$P$ is simple.
        \item If $R\in\cR$ and $n\in\N$, then $F^{-n}(R)$ is simple.
    \end{enumerate}
\end{lemma}
\begin{proof}
    (a) is trivial (using $\Bd(B_1\cap B_2) \subset \Bd(B_1) \cup \Bd(B_2)$).
    For~(b), $F^{-n}(R)$ is clearly open, so we only need show that every
    $(\Bd(F^{-n}(R)))^{(t)}$ is finite. 

    If $(x,t)\in P$, then $(x,t)\not\in F^{-n}(R)$ if and only if $f_t^n(x)$ is
    not in the open interval $R^{(t)}$. If this is the case then the same is
    true for all $(x', t')$ in a neighborhood of~$(x,t)$ --- and hence
    $(x,t)\not\in\Bd(F^{-n}(R))$ --- unless $f_t^n(x)\in\Bd(R^{(t)})$. It
    follows that $\#(\Bd(F^{-n}(R)))^{(t)})\le 2^{n+1}$ for each~$t$.
\end{proof}

\begin{lemma}
\label{Blemma}
    Let $k\ge 1$, $R_1,\dots,R_k\in\cR$, and $n_1,\dots,n_k\in\N$ with $n_k =
    \max(n_1,\dots,n_k)$. Write
    \[
        A = \pi_{n_1}^{-1}(R_1) \cap\dots\cap \pi_{n_k}^{-1}(R_k) \subset \hP.
    \]
    Then $A = \pi_{n_k}^{-1}(B)$ for some simple set~$B$.
\end{lemma}

\begin{proof}
    The proof is by induction on~$k$, with trivial base case $k=1$. For the
    general case, applying the inductive hypothesis to $\pi_{n_2}^{-1}(R_2)
    \cap\dots\cap \pi_{n_k}^{-1}(R_k)$, we can write $A = \pi_{n_1}^{-1}(R_1)
    \cap \pi_{n_k}^{-1}(B')$ for some simple set~$B'$. Observing that if $n\le m$
    and $B_1, B_2\subset P$ we have $\pi_n^{-1}(B_1)\cap
    \pi_m^{-1}(B_2) = \pi_m^{-1}(F^{n-m}(B_1)\cap B_2)$, we can write $A =
    \pi_{n_k}^{-1}(B)$ with $B=F^{n_1-n_k}(R_1) \cap B'$, which establishes
    the result since~$B$ is simple by Lemma~\ref{simple-lem}.
\end{proof}

\begin{theorem}\label{measmain}
    The family of measures $\{\thmut\}$ on $\hP$ is weakly continuous.
\end{theorem}

\begin{proof}
    Let $\cA$ be the collection of all sets~$A$ as in the statement of
    Lemma~\ref{Blemma}: that is, all finite intersections of sets of the form
    $\pi_n^{-1}(R)$, where $R\in\cR$.

    If $A\in\cA$ then, by Lemma~\ref{Blemma}, we can write $A = \pi_n^{-1}(B)$
    for some $n\in\N$ and some simple subset~$B$ of~$P$. Then, if $t_i\to t_*$ is any convergent sequence in~$J$, we have
    \[
        \thmuti(A) = \thmuti(\pi_n^{-1}(B)) = \tmuti(B) \to \tmutst (B) = \thmutst(\pi_n^{-1}(B)) = \thmutst(A)
    \]
    by Remark~\ref{rk1} and weak continuity of the $\tmut$, since $B$ is a
    continuity set for $\tmutst$.

    The result follows from Theorem~\ref{billthm}(b), since $\cA$ is closed
    under finite intersections, and any open subset of $\hP$ is a countable
    union of sets of~$\cA$ (the latter is an immediate consequence of the fact
    that any open subset of $\hP$ is a countable union of sets of the form
    $\pi_n^{-1}(U)$, where $U$ is open in~$P$).
\end{proof}

\section{Weak continuity of the measures $\nu_t$ supported on the
attractors}

Recall from Section~\ref{BBM} that we write $\Pi = D\times J$, $H\colon
\Pi\to\Pi$ for the fat version of the near homeomorphisms~$H_t$, and $\hPi :=
\ilim(\Pi, H)$; and that the parameterized version of Brown's theorem yields a
continuous $h\colon\hPi\to D^2$ which restricts on each slice to a
homeomorphism onto $D^2$.

Since $I\subset D$ we have $P\subset\Pi$, and $H|_P = F$. Therefore
$\hP\subset\hPi$, and the restriction $h|_{\hP}$, which we will also
denote~$h$, acts on each slice of $\hP$ as a homeomorphism onto the
attractor~$\Lambda_t$ of the Barge-Martin homeomorphism $\Phi_t$. The measures
$\nu_t$ supported on these attractors, originally defined by $\nu_t =
(h_t)_*\hmu_t$, can therefore equivalently be defined by $\nu_t = h_*\thmut$.
The following, which is the main result of the first part of the paper, is an
immediate consequence.

\begin{theorem}
\label{main}
    The $\Phi_t$-invariant measures $\nu_t$ on $D^2$ supported on the
    attractors $\Lambda_t$ vary weakly continuously.
\end{theorem} 
\begin{proof}
    If $\alpha\in C(D^2, \R)$ then $\alpha\circ h \in C(\hP, \R)$ and so by
    Theorem~\ref{measmain}, for any $t_i\to t_*$ in $J$,
    \[
        \int_{\hP} \alpha\circ h \; d\thmuti \to
        \int_{\hP} \alpha\circ h \; d\thmutst.
    \]
    Since $\nu_t = h_*\thmut$ we therefore have
    $\displaystyle{
        \int_{D^2} \alpha\; d\nu_{t_i} \to
        \int_{D^2} \alpha \; d\nu_{t_*}
	}$
    as required.
\end{proof}

\section{Physical measures}
\subsection{basic definitions}
Let $M$ be a smooth manifold and $m$ a measure on~$M$ given by a volume form.
It is common to call such a measure ``a Lebesgue measure", and we adopt that
convention.
\begin{definition}
    Let $f:M\to M$ be continuous with invariant Borel probability measure
    $\nu$. A point $x$ is \emph{regular} (or \emph{generic}) for $\nu$
    under~$f$ if
    \begin{equation*}
    	\frac{1}{n} \sum_{i=0}^{n-1} \alpha(f^i(x)) \to
    	\int \alpha\; d\nu
    \end{equation*}
    for all $\alpha\in C(M, \R)$. A collection~$B$ of regular points is called
        a \emph{basin} for $\nu$. The measure $\nu$ is called \emph{physical}
        if it has a basin with positive Lebesgue measure, $m(B)>0$.
\end{definition} 

Note that the set of regular points for $\mu$ under~$f$ is completely invariant
(i.e.\ $x$ is regular if and only if $f(x)$ is regular). The acim~$\mu_t$ of a
tent map $f_t\colon I\to I$ is physical, since the pointwise Birkhoff ergodic
theorem guarantees that $\mu_t$-almost every (and hence Lebesgue almost every)
point is regular for $\mu_t$.

\begin{remark}\label{SRB}
    A measure is physical for one Lebesgue measure if and only if it is physical for all Lebesgue measures, justifying the lack of specificity about the background measure. 

    There is a substantial variance in terminology surrounding physical
    measures. Some authors require a basin to have full measure in an open set.
    Some use \emph{SRB measure} as synonymous with physical measure, others
    regard SRB measures as only existing for smooth systems, where they are
    defined by a condition on disintegration along unstable manifolds;
\textit{cf.} page 741 of \cite{Ysurvey}.
\end{remark}

\subsection{Physical measures for inverse limits}

The fundamental relationship between regular points for a base measure and for
the induced measure in the inverse limit is (\cite{KRS}):
\begin{theorem}[Kennedy, Raines and Stockman] 
\label{KRS} 
    Let $X$ be a compact metric space, $f\colon X\to X$ be continuous and
     surjective, and $\nu$ be an $f$-invariant Borel probability measure. Let
     $\hf\colon\hX\to\hX$ be the natural extension of~$f$, and $\hnu$ be the
     induced $\hf$-invariant measure.
    	\begin{enumerate}[(a)]
            \item A point $x\in X$ is regular for $\nu$ if and only if one (and
                hence every) point of $\pi_0^{-1}(x)$ is regular for $\hnu$.

            \item A set $B$ is a basin for the measure $\nu$ if and only if
            $\pi_0^{-1}(B)$ is a basin for $\hnu$.
    	\end{enumerate}
\end{theorem}
	
\begin{remark}
    If $\bx, \bx'\in \pi_0^{-1}(x)$, then $d(\hf^{\,i}(\bx), \hf^{\,i}(\bx')) =
    d(\bx, \bx')/2^i$. It is therefore immediate that if one point of
    $\pi_0^{-1}(x)$ is regular, then every point of $\pi_0^{-1}(x)$ is regular.
\end{remark}

One would like to be able to show that if $\nu$ is a physical measure for
    $(M,f)$, then $\hnu$ is a physical measure for $(\ilim(M,f), \hf)$. The
    problem, as already mentioned, is that $\ilim(M,f)$ is usually not a
    manifold and so there is no natural way to define a background Lebesgue
    measure. One way forward is via the following reasonable definition
    from~\cite{KRS}:
\begin{definition} 
    Let~$X$ be compact with a Lebesgue measure~$m$, and $f:X\to X$
    be continuous and surjective. An invariant measure~$\hnu$ for the natural
    extension $\hf\colon\hX\to\hX$ is called an \emph{inverse limit physical
    measure} if it has a basin $\hB$ which satisfies $m(\pi_0(\hB))>0$.
\end{definition}
The following is then immediate from Theorem~\ref{KRS}:
\begin{corollary}[Kennedy, Raines and Stockman]
\label{KRS2} 
    Let~$X$ be compact with a Lebesgue measure~$m$, $f\colon X\to X$ be
    continuous and surjective, and $\nu$ be an $f$-invariant Borel probability
    measure. Then $\nu$ is physical if and only if the induced $\hf$-invariant
    measure $\hnu$ is inverse limit physical.
\end{corollary}

\subsection{Inverse limit physical measures for tent maps}
Recall from Section~\ref{BBM} that in the Barge-Martin construction we regard
the interval~$I$ as being contained in the disk~$D$, and construct
near-homeomorphisms $H_t\colon D\to D$ which agree with $f_t$ on~$I$. We will
now regard the acim $\mu_t$ for $f_t$ as a measure on~$D$, supported on~$I$,
and maintain its name. In this section we show that it is a physical measure
for~$H_t$ with respect to the $(y,s)$-coordinate measure --- denoted~$m$ and
called ``Lebesgue'' --- on~$D$. It then follows by Corollary~\ref{KRS2} that
$\hmu_t$, regarded as a measure on $\hD_t$ supported on $\hI_t$, is an inverse
limit physical measure for $\hH_t$.

The Barge-Martin construction depends on the choice of unwrapping $\{\barf_t\}$
of the tent family $\{f_t\}$. From now on we will assume that our unwrapping
satisfies the additional condition that $\barf_t$ restricts to the identity on
the annulus $S\times [0,3/4]\subset D$ for all~$t$. (Such an unwrapping could
be constructed from an arbitrary unwrapping $\{\barg_t\}$ by rescaling
$\barg_t$ to be defined on $S\times[7/8,1]$ and interpolating between the
identity and $\barg_t|_{S\times\{0\}}$ in $S\times [3/4, 7/8]$.)

\begin{theorem}
\label{thm:inv-lim-phys}
    For each~$t$, the measure $\hmu_t$ on $\hD_t$ is an inverse limit physical measure for $\hH_t$.
\end{theorem}

\begin{proof}
    Recall that $\tau\colon S\to I$ is the projection $y\mapsto (y,1)$ from the
    boundary circle of~$D$ onto the interval. For each~$t$, denote by $X_t$ the
    set of regular points for $\mu_t$ under~$f_t$, and let~$Y_t =
    \tau^{-1}(X_t)$. Since $\mu_t$ is ergodic and absolutely continuous with
    respect to Lebesgue, $X_t$ has full Lebesgue measure in~$I$; since $\tau$
    is smooth and at most two to one, $Y_t$ has full angular Lebesgue measure
    $2\pi$ in~$S$.

    Let $Z_t = Y_t\times [1/2, 3/4] \subset D$. If $(y,s)\in Z_t$ then
    $H_t(y,s)= \tau(y)\in X_t$, which is regular for $\mu_t$ under $H_t$  since
    $H_t|_I=f_t$: therefore, by complete invariance of the set of regular
    points, every point of $Z_t$ is regular for $\mu_t$ under $H_t$. Since
    $m(Z_t)>0$, $\mu_t$ is physical for $H_t$, and the result follows from
    Corollary~\ref{KRS2}.
\end{proof}

\section{A background measure for the inverse limit}
\label{sec:background}
The Barge-Martin construction of Section~\ref{BBM} produces a family of
homeomorphisms $\Phi_t\colon D^2\to D^2$ having global attractors $\Lambda_t$,
varying Hausdorff continuously, with $\Phi_t|_{\Lambda_t}$ topologically
conjugate to $\hf_t\colon \hI_t\to\hI_t$. Theorem~\ref{main} states that the
$\Phi_t$-invariant measures $\nu_t$ on~$D^2$ induced from $\hmu_t$ vary weakly
continuously. 

In order to complete the proof of Theorem~\ref{introthm}, we need to show that
the $\nu_t$ are physical measures. While there is no natural connection between
Lebesgue measure on~$D^2$ and any structure on the inverse limits, we will now
use a modification of a construction from~\cite{prime} to obtain a weakly
continuous family of Oxtoby-Ulam measures $\rho_t$ on~$D^2$, with respect to
which the $\nu_t$ are physical.

\subsection{Homeomorphisms onto the complement of the attractors}
Let $A$ be the half-open annulus $A=S\times[0,\infty)$. In this section we will
define a slice-preserving map $\Psi\colon A\times J\to \hPi$, whose image is
the complement of the union of the attractors $\hI_t$, and which is a
homeomorphism onto its image. This homeomorphism will be used to transfer
Lebesgue measure on each slice $A\times\{t\}$ to an Oxtoby-Ulam measure on
$\hD_t$.

We start by defining $\Psi$ on each slice. 
\begin{definition}
    For each~$t\in J$, define $\Psi_t\colon A\to \hD_t\setminus \hI_t$ by
    \[
        \Psi_t(y, s) = 
        \begin{cases}
            \thr{(y, s), (y, s/2), (y, s/4), \dots} & \text{ if }s\in[0,1),\\
            \thr{f_t^{\floor{s}-1}(H_t(y,v)), \dots, f_t(H_t(y,v)), H_t(y, v), (y, v), (y, v/2), \dots} & \text{ otherwise,}
        \end{cases}
    \]
    where, in the second case, $v=(u+1)/2$, with $u=s-\floor{s}$ the
    fractional part of~$s$.
\end{definition}

\begin{lemma}
\label{Psi-bijection}
    Each $\Psi_t$ is a bijection
\end{lemma}
\begin{proof}
    Let $\bz = \thr{z_0, z_1, \ldots}\in \hD_t\setminus\hI_t$, so that each
    $z_n\in D$ and $H_t(z_{n+1})=z_n$ for each~$n$. Note that any point
    $(y,s)\in D\setminus I$ (i.e.\ with $s<1$) has unique $H_t$-preimage $(y,
    s/2)$. Therefore if $z_0=(y,s)\not\in I$, then $\bz=\Psi_t(y, s)$. On the
    other hand, if $z_0\in I$ then, since $\bz\not\in\hI_t$, there is some
    least $k>0$ with $z_k\not\in I$. Writing $z_k=(y, v)$ with $v\in[1/2, 1)$
    we have $\bz=\Psi_t(y, k + 2v - 1)$. Therefore $\Psi_t$ is surjective. 

    Given $\bz=\Psi_t(y,s)$, we can determine the integer part of~$s$ as the
    first $k$ with $z_k\not\in I$; and $y$ and the fractional part of~$s$ from
    $z_k$. Therefore $\Psi_t$ is injective.
\end{proof}

Recall from Section~\ref{BBM} the embeddings $\iota_t\colon\hD_t\to\hPi$
defined by $\bz\mapsto \thr{(z_0, t), (z_1, t),\ldots}$. Write $\hPi_C =
\hPi\setminus\bigsqcup_{t\in J}\iota_t(\hI_t)$, the complement of the union of
the attractors, and define $\Psi\colon A\times J\to \hPi_C$ by $\Psi( (y,s), t)
=
\iota_t(\Psi_t(y,s))$.

\begin{theorem}
\label{thm:Psi-cts}
    $\Psi$ is a slice-preserving homeomorphism. In particular, each $\Psi_t$ is
    a homeomorphism.
\end{theorem}
\begin{proof}
    That $\Psi$ is a slice-preserving bijection is immediate from its
    definition and from Lemma~\ref{Psi-bijection}. The restriction of $\Psi$ to
    $S\times[0,1)\times J$ is given by
    \[
        \Psi( (y, s), t) = \thr{( (y,s), t), ( (y,s/2), t), \dots},
    \]
    which is evidently a homeomorphism onto its image. Let $G\colon A\times
    J\to A\times J$ be the homeomorphism defined by
    \[
        G( (y,s), t) = 
        \begin{cases}
            ( (y, 2s), t) & \text{ if } s\in [0,1],\\
            ( (y, s+1), t) & \text{ if } s\in [1, \infty).
        \end{cases}
    \]  
    It follows immediately from the definitions that $\hH\circ\Psi = \Psi\circ
    G\colon A\times J \to \hPi_C$ (the three cases $s\in[0,1/2)$, $s\in[1/2,
    1)$, and $s\in [1,\infty)$ need to be checked separately). Therefore, for
    each $N\ge 1$, since $G^{-N}$ maps $S\times[0, N+1)\times J$ onto
    $S\times[0, 1)\times J$, we have
    \[
        \Psi|_{S\times[0, N+1)\times J} = \hH^N\circ \Psi|_{S\times[0,1)\times
        J} \circ G^{-N}|_{S\times [0, N+1)\times J},
    \]
    so that $\Psi$ is a homeomorphism on each $S\times [0, N+1)\times J$, and
    hence is a homeomorphism.
\end{proof}

\subsection{The background measure on $\hD_t$}
Let $m$ be the probability measure defined on $A$ by $dm =
\frac{1}{K}\arctan(s)\,dyds$, where $K$ is the normalization constant.
For each~$t\in J$, let $\rho_t'$ be the measure on $\hD_t$ obtained by pushing forward $m$ with $\Psi_t$, and assigning $\rho_t'(\hI_t)=0$.

Recall (see for example~\cite{alpernprasad,fathi}) that a Borel probability
measure on a manifold~$M$ is called \emph{Oxtoby-Ulam (OU)} or \emph{good} if
it is non-atomic, positive on open sets, and assigns zero measure to the
boundary of~$M$ (if it exists).

\begin{theorem}
\label{thm:rho'}
    For each~$t$, $\rho_t'$ is an OU measure on~$\hD_t$ with respect to which
    the set of points which are regular for $\hmu_t$ under~$\hH_t$ has positive
    measure.
\end{theorem}

\begin{proof}
    Because $\Psi_t$ is a homeomorphism, $(\Psi_t)_*(m)$ is an OU measure on
    $\hD_t\setminus \hI_t$. Since $\hI_t$ is closed and nowhere dense in
    $\hD_t$, the extension $\rho'_t$ of this measure is also OU.

    Recall from the proof of Theorem~\ref{thm:inv-lim-phys} that every point of
    $Z_t = Y_t\times[1/2, 3/4]\subset D$ is regular for $\mu_t$ under~$H_t$. It
    follows by Theorem~\ref{KRS} that every point of $R_t':=\pi_0^{-1}(Z_t)$ is
    regular for $\hmu_t$ under $\hH_t$. However, the points of $R_t'$ are
    precisely threads $\thr{(y, s), (y, s/2), \dots}$ with $y\in Y_t$ and
    $s\in[1/2, 3/4]$: that is, $R_t' = \Psi_t(Y_t\times [1/2, 3/4])$. Since
    $Y_t$ has positive one-dimensional Lebesgue measure we have
    $m(Y_t\times[1/2, 3/4])>0$, and the result follows.
\end{proof}

\subsection{The background measures on~$D^2$}
Recall from Section~\ref{BBM} that the parameterized Barge-Martin construction
yields a continuous map \mbox{$h\colon\hPi\to D^2$} which restricts on each
slice to a homeomorphism $h_t\colon \hD_t\to D^2$; that the disk homeomorphisms
$\Phi_t$ are defined by $\Phi_t = h_t\circ\hH_t\circ h_t^{-1}$; and that the
measure $\nu_t$ on $D^2$ is defined by $\nu_t = (h_t)_*\hmu_t$.

Theorem~\ref{main} states that the measures $\nu_t$ vary weakly continuously.
We can now complete the proof of Theorem~\ref{introthm}:

\begin{lemma}
    For each~$t$, the measure~$\nu_t$ is physical with respect to the OU
    measure $\rho_t:=(h_t)_*\rho_t'$. The family of measures $\rho_t$ itself
    varies weakly continuously.
\end{lemma}

\begin{proof}
    Recall from the proof of Theorem~\ref{thm:rho'} that $R_t'
    =\pi_0^{-1}(Z_t)$ consists of regular points for $\hmu_t$ under $\hH_t$,
    and is assigned positive measure by the OU measure $\rho'_t$. Since $h_t$
    is a homeomorphism which conjugates $\hH_t$ and $\Phi_t$, it is immediate
    that $R_t := h_t(R'_t)$ consists of regular points for $\nu_t$ under
    $\Phi_t$, and is assigned positive measure by the OU measure $\rho_t$.

    It therefore only remains to show weak continuity of the family
    $\{\rho_t\}$, so let $t_i\to t_*$ in~$J$. By Theorem~\ref{thm:Psi-cts} we
    have that $h_{t_i}\circ\Psi_{t_i}\to h_{t_*}\circ\Psi_{t_*}$ pointwise.

    Let $\alpha\in C(D^2, \R)$. Since $\rho_t$ is supported on
    $h_t\circ\Psi_t(A)$, we have
    \[
        \int_{D^2} \alpha\,d\rho_{t_i} = 
        \int_{h_{t_i}\circ\Psi_{t_i}(A)} \alpha\, d(h_{t_i}\circ \Psi_{t_i})_*
        m = 
        \int_A \alpha\circ h_{t_i}\circ \Psi_{t_i}\,dm
        \to \int_A \alpha\circ h_{t_0}\circ \Psi_{t_0}\, dm
        =\int_{D^2}\alpha\,d\rho_{t_0}
    \]
    as required, by the bounded convergence theorem.
\end{proof}

\subsection{Using the homeomorphic Measures Theorem}
\label{sec:iso-meas}
For a measure to be physical with respect to a background OU measure is not
entirely satisfactory, since positive measure sets for an OU measure may have
zero Lebesgue measure.

The Homeomorphic Measures Theorem due to Oxtoby and Ulam, and also to von
Neumann, states that, for any OU measure $\rho$ on a manifold~$M$, there is a
homeomorphism $g\colon M\to M$ such that $g_*\rho$ is Lebesgue
measure~\cite{OU,alpernprasad}. In particular, if $\lambda$ is a fixed Lebesgue
measure on~$D^2$ we can find disk homeomorphisms $\Theta_t$ with
$(\Theta_t)_*\rho_t= \lambda$. Then the disk homeomorphisms $\Omega_t =
\Theta_t^{-1}\circ \Phi_t\circ\Theta_t$ have ergodic invariant measures $\xi_t
= (\Theta_t)_*\nu_t$, supported on the global attractors $\Theta_t(\Lambda_t)$;
and the measures $\xi_t$ will be physical with respect to the fixed Lebesgue
measure~$\lambda$.

A natural question is whether, since the $\rho_t$ vary weakly continuously, the
homeomorphisms $\Theta_t$ can be chosen to vary continuously with~$t$: without
this, there is no reason to suppose that the $\Omega_t$ will vary continuously,
or the $\xi_t$ weakly continuously. This is an open question posed by
Fathi~\cite{fathi}, who gives some partial answers: there are also relevant
results due to Peck and Prasad (personal communication). However these results
don't apply in the present context, unless one could succeed in imposing very
particular properties on the family of homeomorphisms~$h_t$.

\section{Application to a family of transitive sphere homeomorphisms}

In~\cite{prime} it is shown how a family of transitive sphere
homeomorphisms $\chi_t\colon S^2\to S^2$ can be constructed from the core tent
maps $f_t$ for $t>\sqrt{2}$. These sphere homeomorphisms are factors of the
natural extensions $\hf_t\colon\hI_t\to\hI_t$ by mild semi-conjugacies. The
following result is a combination of Theorems~5.19 and~5.32 of~\cite{prime}.
\begin{theorem}
\label{thm:semi-conj}
    Let $J=(\sqrt{2}, 2]$. There is a continuously varying
    family~$\{\chi_t\}_{t\in J}$ of self-homeomorphisms of~$S^2$, with each
    $\chi_t$ being a factor of $\hf_t\colon\hI_t\to\hI_t$ by a semi-conjugacy
    $g_t\colon \hI_t\to S^2$, all of whose fibers except perhaps one has at
    most~3 points, and whose exceptional fiber carries no topological entropy.

    Each $\chi_t$ is topologically transitive, has dense periodic points, and
    has topological entropy~$\log t$. Moreover, $\eta_t = (g_t)_*(\hmu_t)$ is
    an ergodic invariant OU measure of maximal entropy.
\end{theorem}

The techniques used in this paper can also be applied to show that the measures
$\eta_t$ vary weakly continuously:

\begin{theorem}
    The measures $\eta_t$ on $S^2$ from the statement of
    Theorem~\ref{thm:semi-conj} vary weakly continuously.
\end{theorem}
\begin{proof}[Sketch proof]
    The details used in this sketch are contained in Section~5.4
    of~\cite{prime}. In that paper, the Barge-Martin construction is carried
    out not in the disk~$D$, but in the sphere~$T$ obtained by collapsing the
    boundary of~$D$, giving rise to Barge-Martin sphere homeomorphisms $\hH_t\colon \hT_t\to\hT_t$ having global attractors $\hI_t$.

    The inverse limit of the fat family of Barge-Martin near homeomorphisms,
     here denoted $\hPi$, is there denoted $\hT_*$; and the commutative diagram
     of Figure~15 of~\cite{prime} includes a slice-preserving map $k =
     K\circ\pi\colon \hT_*\to S^2\times J$, whose restriction $k_t\colon
     \hT_t\to S^2$ to the $t$-slice satisfies $k_t|_{\hI_t} = g_t$, the
     semi-conjugacy of Theorem~\ref{thm:semi-conj}. That is, the
     semi-conjugacies $g_t$ can be gathered into a single continuous map
     $g\colon\hP\to S^2$, with the measures $\eta_t$ given by $\eta_t =
     g_*\thmut$. That these measures vary weakly continuously then follows
     exactly as in the proof of Theorem~\ref{main}.

\end{proof}

\bibliographystyle{amsplain}
\bibliography{weakrefs}

\end{document}